\documentclass[sn-mathphys-num]{sn-jnl}

\makeatletter
\@namedef{subjclassname@1991}{2020 Mathematics Subject Classification}
\makeatother

\setcitestyle{square}

\usepackage{hyperref}
\usepackage[ansinew]{inputenc}
\usepackage{graphicx}
\usepackage{amsmath}
\usepackage{amsthm}
\usepackage{amssymb, color}%
\usepackage{mathrsfs}
\usepackage{enumerate}
\usepackage{bbm}
\usepackage{tikz}
\usepackage{mathptmx,url}

\newcommand{\R}{\mathbb{R}}
\newcommand{\E}{\mathbb{E}}
\newcommand{\N}{\mathbb{N}}

\newcommand{\Z}{\mathbb{Z}}
\newcommand{\VP}{\mathbb{VP}}
\renewcommand{\P}{\mathbb{P}}

\newcommand{\sK}{\mathcal{K}}
\newcommand{\sU}{\mathcal{U}}
\newcommand{\sX}{\mathcal{X}}

\newcommand{\bU}{{\bf U}}
\renewcommand{\bU}{U}
\newcommand{\wtJ}{\widetilde{J}}
\newcommand{\whJ}{X} 
\renewcommand{\epsilon}{\varepsilon}
\newcommand{\eps}{\varepsilon}

\newcommand{\1}{\mathbbm{1}}
\newcommand{\todistr}{\stackrel{d}{\longrightarrow}}

\newcommand{\toprobab}{\overset{\mathbb{P}}{\longrightarrow}}

\theoremstyle{thmstyleone}
\newtheorem{theorem}{Theorem}[section]
\newtheorem{lemma}[theorem]{Lemma}
\newtheorem{corollary}[theorem]{Corollary}
\newtheorem{proposition}[theorem]{Proposition}

\theoremstyle{thmstylethree}
\theoremstyle{definition}

\theoremstyle{thmstyletwo}

\theoremstyle{remark}
\newtheorem{remark}[theorem]{Remark}
\theoremstyle{theorem}
\newtheorem*{Problem}{Open Problem}

\newcommand{\am}[1]{\textcolor{blue}{#1}}
\renewcommand{\am}[1]{#1}

\begin{document}

\title[Set-valued recursions]{Set-valued recursions
  arising from vantage-point trees}

\author[1]{\fnm{Congzao} \sur{Dong}}\email{czdong@xidian.edu.cn}

\author[2,3]{\fnm{Alexander} \sur{ Marynych}}\email{marynych@knu.ua}

\author[4]{\fnm{Ilya} \sur{Molchanov}}\email{ilya.molchanov@unibe.ch}

\affil[1]{\orgdiv{School of Mathematics and Statistics},
  \orgname{Xidian University}, \orgaddress{\street{No. 266 Xinglong Section of Xifeng Road},
    \city{Xi'an}, \postcode{710126}, \state{Shaanxi Province}, \country{China}}}

\affil[2]{\orgdiv{Faculty of Computer Science and Cybernetics},
  \orgname{Taras Shev\-chen\-ko National University of Kyiv},
  \orgaddress{\street{64/13, Volodymyrska Street}, \city{Kyiv}, \postcode{01601}, \country{Ukraine}}}
  
\affil[3]{\orgdiv{School of Mathematical Sciences},
  \orgname{Queen Mary University of London},
  \orgaddress{\street{Mile End Road}, \city{London}, \postcode{E14NS}, \country{United Kingdom}}}

\affil*[4]{\orgdiv{Institute of Mathematical Statistics and Actuarial
    Science}, \orgname{University of Bern},
  \orgaddress{\street{Alpeneggstrasse 22}, \city{Bern},
    \postcode{CH-3012}, \country{Switzerland}}}

\abstract{
  We study vantage-point trees constructed using an independent sample
  from the uniform distribution on a fixed convex body $K$ in
  $(\mathbb{R}^d,\|\cdot\|)$, where $\|\cdot\|$ is an arbitrary
  norm on $\mathbb{R}^d$. We prove that a sequence of
  sets, associated with the left boundary of a
  vantage-point tree, forms a recurrent Harris chain on the space of
  convex bodies in $(\mathbb{R}^d,\|\cdot\|)$. The limiting
  object is a ball polyhedron, that is, an a.s.~finite intersection of
  closed balls in $(\mathbb{R}^d,\|\cdot\|)$ of possibly different
  radii. As a consequence, we derive a limit theorem for the length of 
  the leftmost path of a vantage-point tree.
}

\keywords{compact convex sets, Harris chain, Random tree,
  Vantage-point tree, VP tree, ball polyhedron}



\pacs[Mathematics Subject Classification]{60C05, 60J05, 68P10}

\maketitle

\newpage

\section{Introduction}
\label{sec:introduction}

Let $(M,\rho)$ be a metric space. The notation $B_{r}(x)$
is used for the closed ball of radius $r$ centered at $x$, that is,
$B_{r}(x):=\{y\in M:\rho(x,y)\leq r\}$. {\it A vantage-point tree} (in
short, vp tree) $\VP(\sX)$ of a (finite or infinite) sequence
$\sX:=(x_1,x_2,\ldots)\subset M$ with a threshold function
$r$ is a labeled rooted subtree of a full binary tree constructed
using the following rules.
\begin{itemize}
\item Each vertex of $\VP(\sX)$ is a pair \am{$(x,r)$}, where $x\in\sX$
  is called a {\it vantage-point} and its label \am{$r$} is
  a positive real number called the {\it threshold} of $x$. 
\item $\VP(x_1)$ is the unique tree with a single vertex (the
  root) $(x_1,r)$, where \am{$r=r_{x_1}$} is a given positive number, 
  the threshold of the \am{root, which might depend on $x_1$.}
\item For a finite set $(x_1,\dots,x_k)\subset M$, the tree
  $\VP(x_1,\dots,x_k,x)$ is constructed by adding a new vertex
  \am{$(x,r)$} to $\VP(x_1,\dots,x_k)$ by recursively comparing $x$
  with $x_1,\dots,x_k$, starting from its root $x_1$ and according to
  the procedure: if $x\in B_{r_y}(y)$, where $y$ is one of the points
  $x_1,\ldots,x_k$ \am{and $r_y$ is its threshold}, then $x$ goes to the left subtree of $(y,r_y)$;
  and to the right subtree if $x\notin B_{r_y}(y)$. If
  $x\in B_{r_y}(y)$ and the left subtree of $y$ is empty, then
  \am{$(x,r)$} is attached as the left child to $(y,r_y)$, whereas if
  $x\notin B_{r_y}(y)$ and the right subtree of $y$ is empty, then
  \am{$(x,r)$} is attached as the right child to $(y,r_y)$. Finally, the
  threshold value $r$ of \am{$x$} is determined according to the chosen rule \am{and, in general, 
  is a function of both $x$ and the previously constructed tree $\VP(x_1,\dots,x_k)$.}
\end{itemize}

Vantage-point trees were introduced in~\cite{Yianilos:1993} as a data
structure for efficient storing and retrieving spatial data,
particularly, for fast execution of nearest-neighbor search queries in
a metric space~\cite{Fu_et_al,Nielsen+Piro+Barlaud:2009}.
There are several close relatives of vp trees such as
kd-trees~\cite{Bentley,Moore} and ball trees~\cite{Omohundro}.

The choice of a threshold \am{$r$} for a newly added vertex \am{$(x,r)$} is
a part of specification of a vantage-point tree and usually depends on
the position of a vertex in the already constructed tree to which
\am{$(x,r)$} is attached. This choice is usually dictated by the
requirement that $\VP(x_1,\ldots,x_k)$ remains balanced when
$k\to\infty$.  In many cases and, in particular, if the points in
$\sX$ are `uniformly' scattered in some compact subset of $M$, it is
natural to assume that \am{$r=r_x$} decreases exponentially fast as a
function of the depth of $x$, that is,
\begin{equation}
  \label{eq:r_x_exponential}
  r_x~=c\cdot \tau^{{\rm depth}(x)+1},
\end{equation}
for some $\tau\in(0,1)$ and $c>0$, where ${\rm depth}(x)$ is the
distance from $x$ to the root. Without loss of generality we set
$c:=1$. The results for a general $c$ can be obtained
by scaling the metric, so that all results hold with the unit ball
replaced by the ball of radius $c$. Of course, the shape of the vp
tree heavily depends on the choice of $\tau$.

Assume that $M$ is a \emph{convex body} $K$ (a convex compact set with
non-empty interior) in Euclidean space $\R^d$ endowed with an
arbitrary 
norm $\|\cdot\|$, which is used to
construct balls appearing in the definition of the vp tree. In this paper
we focus on a particular class of vp trees constructed using an
independent sample from the uniform distribution on $K$. Assume that
$\sX:=(\bU_1,\bU_2,\ldots)$ for a sequence $(\bU_j)_{j\in\N}$ of
independent copies of a random vector $\bU$ with distribution
\begin{equation}
  \label{eq:uniform_def}
  \P\{\bU\in\,\cdot\, \}
  =\frac{\lambda(\,\cdot\, \cap K)}{\lambda(K)},
\end{equation}
where $\lambda$ is the Lebesgue measure in $\R^d$. \am{Here and in what follows we assume that the sequence $(\bU_1,\bU_2,\ldots)$ and the vector $U$ are defined on some underlying probability space $(\Omega,\mathcal{F},\mathbb{P})$.}

Recapitulating, in this paper we will consider an infinite vp tree
$\VP(\bU_1,\ldots,\bU_n,\ldots)$ constructed
from independent identically distributed random vectors having uniform
distribution~\eqref{eq:uniform_def} and with the threshold function
given by~\eqref{eq:r_x_exponential}. It is natural to call such a tree
{\it random vp tree with an exponential threshold function}. To the
best of our knowledge,~\cite{Bohun} is the only paper devoted to the
probabilistic analysis of such trees, which restricts the study to vp
trees in $K=[-1,1]^d$ with the $\ell_\infty$-norm.

As we see in Section~\ref{sec:conv-rand-sets} below, the
analysis of the leftmost path in a random vp tree with an exponential
threshold function leads to a set-valued recursion of the form
\begin{equation}\label{eq:basic_recursion_intro}
  \whJ_{h+1}=\tau^{-1}(\whJ_{h}-u_h)\cap B_1, \quad h\in \N_0,
\end{equation}
where $u_h$ is a point uniformly sampled from $\whJ_h$ and
$B_1:=B_1(0)$ and $\N_0:=\{0,1,2,\ldots\}$.  We study the sequence
$(\whJ_{h})_{h\in\N_0}$, which forms a set-valued Markov chain on the
family $\sK_d$ of convex bodies. The importance of the sets $\whJ_{h}$
lies in the fact that they describe the basins of attraction for the
successive vertices which can be attached to the leftmost path.
Our main result shows that $\whJ_h$ has a limit distribution and the
limiting random set is obtained as the intersection of a random number
$m$ of unit balls scaled by $1,\tau^{-1},\dots,\tau^{-m+1}$. To the best
of our knowledge, set-valued Markov chains have not been
systematically investigated in the literature and we are aware only of
several `genuinely' set-valued Markov chains\footnote{We took some
  liberty to use adjective `genuinely' to outline set-valued Markov
  chains whose analysis cannot be easily reduced to the study of
  Markov chains with finite-dimensional state spaces.} studied
before, namely, continued fractions on convex
sets~\cite{Molchanov} and diminishing process of B\'{a}lint T\'{o}th~\cite{Kevei+Vigh}. 

The paper is organized as follows. In Section~\ref{sec:conv-rand-sets}
we discuss in details the origins of Markov
chain~\eqref{eq:basic_recursion_intro} in relation to random vp trees
with exponential threshold functions and formulate the main result as
Theorem~\ref{thm:main1}. The proof of the main result is presented in
Section~\ref{sec:proof}. In Section~\ref{sec:leftmost} we apply
Theorem~\ref{thm:main1} to derive a limit theorem for the length of
the leftmost path of a vp tree with an exponential threshold function.



\section{Convergence of random sets associated with vp trees}
\label{sec:conv-rand-sets}

There is a natural sequence of nested sets associated with the left
boundary of an arbitrary vp tree $\VP(x_1,x_2,\ldots)$, that is, with
the unique path which starts at the root and on each step follows the
left subtree. Let $(x_{l_h},r_{x_{l_h}})$ be the vertex of depth
$h\in\N_0$ in the aforementioned unique path. Thus, $x_{l_0}=x_1$ is
the root, $x_{l_1}$ is the unique left child of the root, $x_{l_2}$ is
the unique left child of $x_{l_1}$, and so on. Define the sequence of
nested convex closed sets $(I_{h})_{h\in\N_0}$ recursively by the rule
\begin{equation}
  \label{eq:I_h_def}
  I_{h+1}:=I_{h}\cap B_{r_{x_{l_h}}}(x_{l_h}),\quad h\in \N_0,
\end{equation}
with $I_0:=K$ being a fixed convex body, see
Section~\ref{sec:introduction}. Recall that $B_r(x)$ denotes the ball
of radius $r$ in the chosen norm on $\R^d$ centered at $x$. In the
following we write $B_r$ for the ball $B_r(0)$ centered at the origin.

Thus, $I_{h+1}$ is a subset of $K$
such that a point landing there goes to the left subtree of
$(x_{l_h},r_{x_{l_h}})$.
By the construction we know that $x_{l_h}$ is the unique left child of
$x_{l_{h-1}}$, and thereupon $x_{l_h}\in I_{h}$.

Specifying~\eqref{eq:I_h_def} to the random vp tree with an
exponential threshold function, we see that $r_{x_{l_h}}=\tau^{h+1}$,
$h\in\N_0$. Furthermore, it is clear that the conditional distribution
of $x_{l_h}$, given $I_{h}$, is uniform in $I_{h}\subset K$. Thus, for
an arbitrary Borel set $A$, we have
\begin{equation*}
  \P\{x_{l_h}\in A \,|\, I_{h}\}
  =\frac{\lambda(A\cap I_{h})}{\lambda(I_{h})},\quad h\in \N_0.
\end{equation*}

\am{Throughout this paper, we will often need to work with points sampled uniformly at random from a varying convex body $L\subseteq K$. To address this, we introduce a random mapping $\sU$ (\am{a measurable function $\sU:\sK_d\times \Omega\to\mathbb{R}^d$}) that assigns to each deterministic convex body $L\in\sK_d$, where $L\subseteq K$, a random point $\sU(L)\in L$ with the uniform distribution on $L$, that is,
\begin{equation}\label{eq:U_mapping_marginal}
  \P\{\sU(L)\in\cdot\}=\frac{\lambda(\cdot\cap
    L)}{\lambda(L)},
  \quad L\in\sK_d,\quad \am{L\subseteq K.}
\end{equation}
Of course, there are many ways to construct the mapping $\sU$ resulting in various finite-dimensional distributions of the random field $(\sU(L))_{L\in\sK_d}$. For example, a simple construction hinges on a sequence $(\widetilde{\bU}_1,\widetilde{\bU}_2,\ldots)$ of independent uniform points in $K$ and is obtained by putting
$$
\sU(L):=\widetilde{\bU}_{\eta_L},\quad\text{where}\quad \eta_L:=\min\{k\geq 1:\widetilde{\bU}_k\in L\},\quad L\in\sK_d,\quad L\subseteq K.
$$
However, for the purpose of the present paper only the marginal distributions specified by~\eqref{eq:U_mapping_marginal} are of importance.}

Given a sequence $(\sU_k)_{k\in\N_0}$ of independent copies of the
mapping $\sU$ we define a Markov chain
$(J_h,y_h)\in \sK_d\times \R^d$, $h\in\N_0$, as follows:
$(J_0,y_0)=(K,\sU_0(K))$ and
\begin{displaymath}
  (J_{h},y_{h})\longmapsto \Big(J_{h}\cap
  B_{\tau^{h+1}}(y_h),\sU_{\am{h+1}}\big(J_{h}
  \cap B_{\tau^{h+1}}(y_h)\big)\Big)=:(J_{h+1},y_{h+1}),\quad h\in\N_0.
\end{displaymath}
The following is a simple observation.

\begin{proposition}\label{prop:markovian_redef}
  Let $(I_h,x_{l_h})_{h\in\N_0}$, $I_0=K$, be a sequence of
  sets~\eqref{eq:I_h_def} constructed from a random vp tree. Then
  the sequences $(I_h,x_{l_h})_{h\in\N_0}$ and $(J_h,y_h)_{h\in\N_0}$
  have the same distribution. For every $h\in\N_0$,
  $y_{h}=\sU_h(J_{h})$, where $\sU_h$ and $J_{h}$ are
  independent. In particular, the conditional distribution of $y_h$,
  given $J_{h}$, is uniform on $J_{h}$.
\end{proposition}

We are interested in the asymptotic shape of the random set $J_h$, as
$h\to\infty$. To this end, we first consider a shifted version of the
sequence $(J_h)$ by setting $\wtJ_h:=J_h-y_{h-1}$ for
$h\in\N_0$, where $y_{-1}:=0$. It can be readily checked that
\begin{equation*}
  \wtJ_{h+1}=\big(\wtJ_{h}-(y_{h}-y_{h-1})\big)
  \cap B_{\tau^{h+1}}\quad\text{for}\quad h\in\N_0.
\end{equation*}
Furthermore, for all Borel $A\subseteq \mathbb{R}^d$,
\begin{align*}
  \P\{y_{h}-y_{h-1}\in A\,|\, \wtJ_0,\ldots,\wtJ_h\}
  &=\P\{\sU_h(J_{h})-y_{h-1}\in A\, |\, \wtJ_0,\ldots,\wtJ_h\}\\
  &=\P\{\sU_h(J_{h}-y_{h-1})\in A\, |\, \wtJ_0,\ldots,\wtJ_h\}
  =\P\{\sU_h(\wtJ_{h})\in A\, |\, \wtJ_h\},
\end{align*}
where we have used that $\sU(K)+x\overset{d}{=}\sU(K+x)$, for every
$x\in\R^d$ and also independence between $\sU_h$ and
$(\wtJ_0,\ldots,\wtJ_h)$. Thus, $(\wtJ_h)_{h\in\N_0}$ is a Markov chain
with the transition mechanism
\begin{equation}
  \label{eq:J_h_def2}
  \wtJ_0=K,\quad \wtJ_h\longmapsto \big(\wtJ_h-\sU_h(\wtJ_{h})\big)\cap
  B_{\tau^{h+1}}
  =\wtJ_{h+1},\quad h\in\N_0.
\end{equation}
The advantage of this chain in comparison to the chain
$(J_h,y_h)_{h\in\N_0}$ is that
$\P\{0\in \wtJ_{h}\subset B_{\tau^h}\}=1$, for all $h\in\N$. Since
$\tau^h\to 0$ as $h\to\infty$, the later implies that $\wtJ_h$
a.s.~converges in the Hausdorff metric on $\sK_d$ to the set $\{0\}$
exponentially fast. Our main result shows that the sequence of {\it
  normalized} sets
\begin{displaymath}
  \whJ_h:=\tau^{-h}\wtJ_h, \quad h\in\N_0,
\end{displaymath}
converges to a non-degenerate distribution.  This sequence
$(\whJ_h)_{h\in\N_0}$ is a time-homogeneous Markov chain with the
transition mechanism
\begin{equation}
  \label{eq:J_h_defX}
  \whJ_0:=K,\quad \whJ_h\longmapsto
  \tau^{-1}\big(\whJ_h-\sU_h(\whJ_{h})\big)\cap B_1=\whJ_{h+1},
  \quad h\in\N_0.
\end{equation}
Thus, we recover recursion~\eqref{eq:basic_recursion_intro} from the
introduction.  Note that, by the construction,
$\P\{\whJ_h\in\sK_d\}=1$ for all $h\in\N_0$.

Let $\sK_d^{(o)}$ be the family of convex bodies which contain the
origin in the interior.

\begin{theorem}\label{thm:main1}
  \am{Assume that~\eqref{eq:r_x_exponential} holds for some
  $\tau\in (0,\,1)$ and $c>0$. Let $(\whJ_h)_{h\in\N_0}$ be a Markov
  chain on $\sK_d$ given by~\eqref{eq:J_h_defX}, where $K\in\sK_d$ is
  an arbitrary convex body. 
  \begin{enumerate}
  \item The sequence $(\whJ_h)_{h\in\N_0}$ is a 
  recurrent aperiodic Markov chain with a unique stationary distribution $\pi_{\infty}$.
\item The sequence of distributions
  $(\mathbb{P}\{\whJ_{h}\in \cdot\})_{h\in\N_0}$ converge, as
  $h\to\infty$, in the total variation distance to $\pi_{\infty}$. In
  particular, if $X_{\infty}$ is a random element taking values in
  $\sK_d$ with the distribution $\pi_{\infty}$, then
  \begin{equation*}
    \whJ_{h}\todistr X_{\infty}\quad \text{as}\quad h\to\infty
  \end{equation*}
  in $\sK_d$ endowed with the Hausdorff metric. 
  \item The random compact convex set $X_{\infty}$ with probability one takes values in $\sK_d^{(o)}$ and is an a.s.~finite intersection of translated and scaled by $\tau^{-j}$, $j=0,\dots,m$, copies of
    the ball $B_1$, where $m$ is random. The distribution of $X_{\infty}$ does not depend on $K$.
  \end{enumerate}}
\end{theorem}

\begin{remark}
Passing from random sets to their equivalence classes up to
translations, we see that the equivalence class of $\tau^{-h}I_h$ 
converges in distribution to the equivalence class of $X_{\infty}$.
\end{remark}

\section{Proof of Theorem~\ref{thm:main1}}\label{sec:proof}

\am{As has been mentioned above, the sequence $(\whJ_{h})_{h\in\N_0}$ is a Markov chain on the state space $\sK_d$.}
The main idea of the proof lies in showing that the chain $(\whJ_{h})_{h\in\N_0}$ visits the state $B_1$ infinitely often with independent identically distributed integrable times between consecutive visits. Intuitively this implies that the chain possesses a unique stationary distribution. \am{To formalize this intuition, and considering that the state space $\sK_d$ is uncountable, we will first verify that $(\whJ_{h})_{h\in\N_0}$ is a {\it Harris chain} on $\sK_d$. We will then demonstrate that it is aperiodic and positive recurrent. Ultimately, our Theorem~\ref{thm:main1} will follow as a direct application of the convergence theorem for aperiodic recurrent Harris chains, see Theorem 6.8.8 in~\cite{Durrett}.}

Recall, see Section 6.8 in~\cite{Durrett}, that a Markov chain
$(\xi_n)_{n\in\N_0}$ on a \am{state space $\mathfrak{S}$ with a $\sigma$-algebra $\mathcal{S}$} is a Harris
  chain if there exist two sets $A,A'\in\mathcal{S}$, a function $q$
such that $q(x, y) \geq \eps > 0$ for $x\in A$, $y\in A'$, and a
probability measure $\rho$ concentrated on $A'$ so that
\begin{itemize}
\item[(i)] for all \am{$z\in\mathfrak{S}$}, we have
  $\P\{\kappa_A<\infty\,|\,\xi_0=z\}>0$, where
  $\kappa_A:=\inf\{n\geq 0:\xi_n\in A\}$;
\item[(ii)] for $x\in A$ and $C\subset A'$,
  $\P\{\xi_{n+1}\in C\,|\, \xi_n=x\}\geq \int_C q(x,y)\rho({\rm d}y)$.
\end{itemize}





\begin{lemma}\label{lem:harris}
  The sequence $(\whJ_{h})_{h\in\N_0}$ is a Harris chain on the state
  space $\sK_d$ with $A=A'=\{B_{\am{1}}\}$, $\rho$ being a degenerate
  probability measure concentrated at $B_1$ and $q$ being a
  constant $(1-\tau)^d>0$.
\end{lemma}
\begin{proof}[Proof of Lemma~\ref{lem:harris}]
In what follows we denote the Minkowski sum of two sets in $\R^d$
by
\begin{displaymath}
  A_1+A_2=\{x+y: x\in A_1, y\in A_2\}
\end{displaymath}
and the Minkowski difference by
\begin{displaymath}
  A_1\ominus A_2:=\{x\in\R^d:x+A_2\subset A_1\}.
\end{displaymath}
In particular, for closed balls $B_r(x)$ and $B_R(y)$ in $\R^d$ \am{(endowed with an arbitrary norm), with}
$R\geq r\geq 0$, we have $B_R(y)\ominus
B_r(x)=B_{R-r}(y-x)$. Furthermore, $A\ominus B_r\neq\varnothing$ means
that $A$ contains a translation of $B_r$. 

\am{Proceeding to the proof of the lemma,} note that 
  \begin{equation}\label{eq:aperiodicity}
    \P\{\whJ_{h+1}=B_1\,|\,\whJ_{h}=B_1\}
    =\P\{\sU(B_1)\in B_{1-\tau}\}
    =\frac{\lambda(B_{1-\tau})}{\lambda(B_1)}=(1-\tau)^d>0.
  \end{equation}
  Thus, part (ii) of the definition holds with $\rho(\{B_1\})=1$
  and $q(x,y)=(1-\tau)^d>0$. To check part (i) we argue as
  follows. For every $K\in\sK_d$ there exist $\epsilon>0$ and
  $x\in K$ such that $B_{\epsilon}(x)\subset K$. Note that
  \begin{displaymath}
    \P\{\whJ_{1}\supset B_{\epsilon/2}\,|\,\whJ_0=K\}
    \geq \P\{\sU_1(K)\in B_{\epsilon/2}(x)\}>0.
  \end{displaymath}
  Thus, without loss of generality we may assume that the chain starts
  at $\whJ_{0}$ which contains a small ball around the origin. We
  now show that with positive probability the chain reaches the
  state $B_1$. Intuitively this occurs whenever we have a
  relatively long series of consecutive events ``a uniform point
  chosen from $\whJ_h$ falls near the origin''. In this case $\whJ_{h+1}$
  contains a scaled copy of $\whJ_h$ with the scale factor greater
  than $1$. To make this intuition precise, note that, for
  $R\in (0,1)$,
  \begin{displaymath}
    \whJ_h\supset B_R\quad\text{and}\quad \sU_h(\whJ_h)\in B_{R(1-\tau)/2}
  \end{displaymath}
  together imply 
  \begin{displaymath}
    \whJ_{h+1}\supset\tau^{-1}(B_R\ominus B_{R(1-\tau)/2})\cap B_1 = B_{R_1},
  \end{displaymath}
  where $R_1:=(R(1+\tau)/(2\tau))\wedge 1>R$. Thus, given that
  $\whJ_h$ contains $B_R$ there is an event of positive probability
  $\{\sU_h(\whJ_h)\in B_{R(1-\tau)/2}\}$ such that $\whJ_{h+1}$
  contains either $B_1$ (and in this case it is equal to $B_1$)
  or contains a scaled copy of $B_R$ with the scale factor
  $(1+\tau)/(2\tau)>1$. This clearly implies (i).
\end{proof}

\begin{remark}
  The argument used in the proof of Lemma~\ref{lem:harris} is a
  simplified version of a much stronger claim that the chain which
  starts at $B_1$ returns to this state with probability one (not just
  with positive probability), and, furthermore, the mean time between
  the visits has finite mean. This claim (confirmed in
  Proposition~\ref{prop:pos_recurrence}) is illustrated on
  Figure~\ref{fig:j_h_chain} for 
  $\tau=4/7$. On each step a uniform random point inside a current set
  $\whJ_h$ is picked, the set is translated by the chosen vector,
  scaled by $\tau^{-1}>1$ and intersected with the unit disk. The
  chain returns to the state $B_1$ on Step 15.
\end{remark}

\begin{figure}
  \begin{center}
    \includegraphics[scale=0.4]{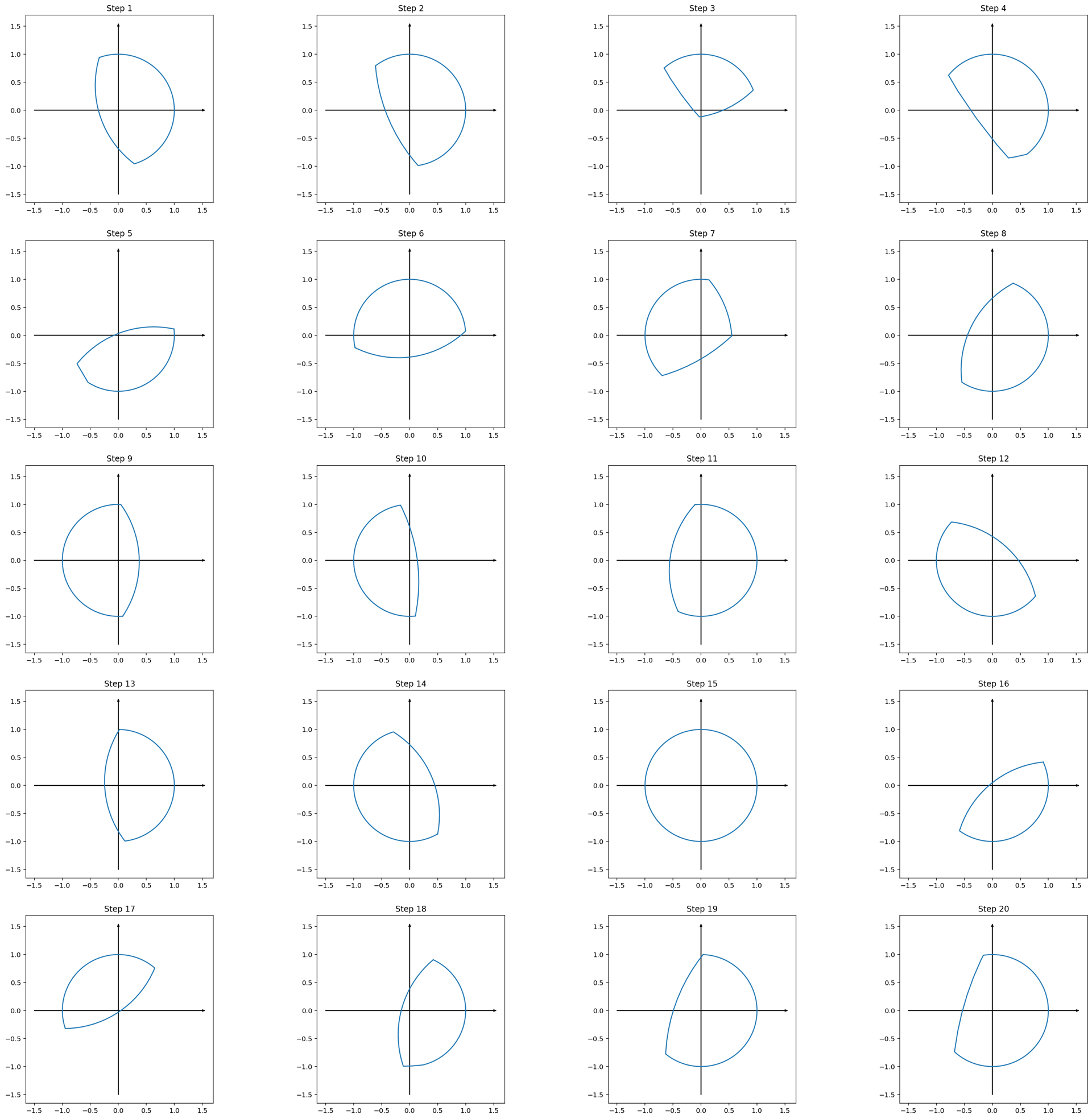}
  \end{center}
  \caption{First twenty values of the chain $(\whJ_{h})_{h\in\N_0}$ in
    $\R^2$ with the Euclidean norm and $\tau=4/7$. The chain starts at
    $B_1$ and returns to this state on Step 15. Each value of the
    chain is a finite intersection of translated and scaled unit
    balls.}
    \label{fig:j_h_chain}
\end{figure}

The following observation is crucial for the proof that $(\whJ_h)$ is
positive recurrent. It tells us that with probability one $\whJ_h$
contains a ball of a small (but fixed) radius, for all sufficiently
large $h\in\N_0$.

\begin{proposition}\label{prop:contains_a_ball}
  Let $K\in\sK_d$ be an arbitrary compact convex body which contains a
  ball of radius $\eps>0$. Put $r:=\min(\eps,1) 2^{-d-1}$. Then
  \begin{equation*}
    \P\{\whJ_h\ominus B_{r}\neq \varnothing
    \text{ for all }h\in\N_0| \whJ_0=K\}=1.
  \end{equation*}
\end{proposition}
\begin{proof}
  Without loss of generality assume that $K$ contains the origin in
  its interior and $B_{\eps}\subset K$. Moreover, starting with
  $\whJ_1$, we can assume that $K\subset B_1$. For notational
  simplicity put $\sU_k(\whJ_k)=:u_k$, $k\in\N_0$, so, let us repeat
  again,
  \begin{equation*}
    \whJ_0=K,\quad \whJ_{h+1}=\tau^{-1}(\whJ_{h}-u_h)\cap B_1,\quad h\in\N_0.
  \end{equation*}
  By induction,
  \begin{equation}
    \label{eq:contains_a_ball_main_representation}
    \whJ_{h}=\left(\tau^{-h} K-\sum_{j=1}^h \tau^{-j}u_{h-j}\right)
    \bigcap \left(\bigcap_{k=0}^{h-1}\left(B_{\tau^{-k}}-\sum_{j=1}^k
        \tau^{-j}u_{h-j}\right)\right)
    =:\widehat{K}^{(h)}\cap \widehat{B}^{(h)},\quad h\in\N_0.
  \end{equation}
  Note that if the upper index is strictly smaller than the
  lower one, then intersections are set to be equal to $\R^d$
  and sums are set to vanish. We show that 
  \begin{equation*}
    (\widehat{K}^{(h)}\cap \widehat{B}^{(h)})\ominus B_r\neq \varnothing,\quad h\geq d.
  \end{equation*}
  Fix any $h\geq d$. Upon multiplying by $\tau^h$, this is equivalent
  to
  \begin{equation}
    \label{eq:contains_a_ball_proof3}
    \varnothing\neq
    \left(K\ominus B_{\tau^hr}-\sum_{j=0}^{h-1}\tau^ju_j\right)
    \bigcap \left(\bigcap_{k=0}^{h-1}\left(B_{\tau^{h-k}-\tau^{h}r}
    -\sum_{j=h-k}^{h-1} \tau^{j}u_{j}\right)\right)\\
    =:\bigcap_{k=0}^{h} \Big(L^{(k)}-v_{h,k}\Big),
  \end{equation}
  where $L^{(0)}:=K\ominus B_{\tau^hr}$ and
  $L^{(k)}:=B_{\tau^k-\tau^{h}r}$ for $k=1,\dots,h$, and 
  \begin{displaymath}
    v_{h,k}:=\sum_{j=k}^{h-1} \tau^{j}u_{j}, \quad k=0,\dots,h.
  \end{displaymath}
  
  Since $u_h\in \whJ_h\subset \tau^{-h} K-\sum_{j=1}^h \tau^{-j}u_{h-j}$ and $\tau^h u_h+v_{h,0}=v_{h+1,0}$,
  we have
  \begin{equation}
    \label{eq:2}
    v_{h,0}\in K\subset B_1, \quad h\in\N.
  \end{equation}
  Furthermore,
  \begin{displaymath}
    \am{\tau^hu_h\in B_{\tau^k}-v_{h,k},}
  \end{displaymath}
  so that $\|v_{h,k}\|\leq \tau^k$ for all $k\leq h$. 
  Furthermore, 
  \begin{equation}\label{eq:contains_a_ball_proof62}
    \|v_{h,k}-v_{h,l}\|=\|v_{l,k}\|\leq \tau^{k},\quad 1\leq k<
    l\leq h,\quad h\in\N.
  \end{equation}
  
  In order to check \eqref{eq:contains_a_ball_proof3}, we employ
  Helly's theorem, see \cite[Theorem~I.4.3]{Barvinok}, which tells
  us that the intersection of a finite family of convex sets in $\R^d$
  is non-empty if an intersection of any $d+1$ sets in this family is
  non-empty. Fix $0\leq i_0<i_1<\cdots <i_d\leq h$, put
  \begin{displaymath}
    \lambda_j:=\frac{2^j\tau^{-i_j}}{\sum_{j=0}^d 2^j\tau^{-i_j}},\quad j=0,\ldots,d,
  \end{displaymath}
  \am{and note that $\sum_{j=0}^{d}\lambda_j=1$.}
  We aim to show that 
  \begin{displaymath}
    -\sum_{k=0}^d \lambda_k v_{h,i_k}
    \in\bigcap_{j=0}^{d} \Big(L^{(i_j)}-v_{h,i_j}\Big),\quad h\geq d+1.
  \end{displaymath}
  If $i_0\geq1$, it suffices to check that
  \begin{equation}\label{eq:contains_a_ball_proof6}
    \left\| \sum_{k=0}^d \lambda_k v_{h,i_k} - v_{h,i_j}\right\|
    \leq \tau^{i_j}-r\tau^h,\quad j=0,\ldots,d,\quad h\geq d+1.
  \end{equation}
  By \eqref{eq:contains_a_ball_proof62},
  \begin{displaymath}
    \|v_{h,i_k}-v_{h,i_j}\|\leq \tau^{\min(i_k,i_j)}.
  \end{displaymath}
  Thus, for every fixed $j=0,\ldots,d$,
  \begin{align*}
    \left\| \sum_{k=0}^d \lambda_k v_{h,i_k} - v_{h,i_j}\right\|
    &\leq\sum_{k=0}^{j-1}\lambda_k \|v_{h,i_k}-v_{h,i_j}\|
      +\sum_{k=j+1}^{d}\lambda_k \|v_{h,i_k}-v_{h,i_j}\|\\
    &\leq \frac{1}{\sum_{k=0}^d 2^k\tau^{-i_k}}
      \left(\sum_{k=0}^{j-1}2^k\tau^{-i_k}\tau^{i_k}+\tau^{i_j}\sum_{k=j+1}^{d}2^k\tau^{-i_k}\right)\\
    &=\frac{1}{\sum_{k=0}^d 2^k\tau^{-i_k}}\left(2^j-1+\tau^{i_j}\sum_{k=j+1}^{d}2^k\tau^{-i_k}\right),
  \end{align*}
  where the last sum in parentheses is zero if $j=d$.  This
  estimate demonstrates that~\eqref{eq:contains_a_ball_proof6} is a
  consequence of
  \begin{equation}\label{eq:contains_a_ball_proof8}
    2^j+r\tau^{h}\sum_{k=0}^{d}2^k\tau^{-i_k}
    \leq 1 + \tau^{i_j}\sum_{k=0}^{j}2^k\tau^{-i_k},\quad j=0,\ldots,d,\quad h\geq \am{d+1}.
  \end{equation}
  It remains to note that we have chosen $r\leq 2^{-d-1}$, so that
  \begin{displaymath}
    2^j+r\tau^{h}\sum_{k=0}^{d}2^k\tau^{-i_k}\leq  2^j +
    r\sum_{k=0}^{d}2^k
    \leq 2^j + 1
    \leq 1+\tau^{i_j}\sum_{k=0}^{j}2^k\tau^{-i_k},
  \end{displaymath}
  which implies~\eqref{eq:contains_a_ball_proof8}.

  Now assume that $i_0=0$. Then~\eqref{eq:contains_a_ball_proof6}
  holds for $j=1,\dots,d$ and we need to consider
  \begin{equation}
    \label{eq:3}
    -\sum_{k=0}^d \lambda_k v_{h,i_k} + v_{h,0}
    =\sum_{k=1}^{d}\lambda_k (v_{h,0}-v_{h,i_k})
    =\sum_{k=1}^{d}\lambda_k v_{i_k,0}.
  \end{equation}
  Since $v_{i_k,0}\in K$ for all $k\geq 1$ by~\eqref{eq:2}, the sum on the
  right-hand side belongs to $(1-\lambda_0)K$. Thus, the left-hand
  side of \eqref{eq:3} belongs to $L^{(0)}=K\ominus B_{\tau^hr}$ if 
  \begin{displaymath}
    (1-\lambda_0)K+ B_{\tau^hr}\subset K,
  \end{displaymath}
  equivalently,
  \begin{equation}
    \label{eq:1}
    B_r\subset\lambda_0 \tau^{-h} K.
  \end{equation}
  Since 
  \begin{displaymath}
    \lambda_0 \tau^{-h} K=\frac{\tau^{-h}}{\sum_{j=0}^d 2^j \tau^{-{i_j}}} K
    =\frac{1}{\sum_{j=0}^d 2^j \tau^{h-{i_j}}} K\supset 2^{-(d+1)}K,
  \end{displaymath}
  \eqref{eq:1} holds for $r=\eps 2^{-d-1}$ if $K$ contains $B_\eps$.

  If $h<d$, we add fictitious balls $B_1$ to the
  intersection~\eqref{eq:contains_a_ball_main_representation}, note
  that $v_{h,k}=0$ for $k\geq h+1$, and repeat the arguments.
\end{proof}

Proposition~\ref{prop:contains_a_ball} is essential to prove the
following result, which shows that $(\whJ_h)_{h\in\N_0}$ is a
recurrent Harris chain.

\begin{proposition}\label{prop:pos_recurrence}
  Assume that $\whJ_0=K$ for some $K\in\sK_d$. Let
  \begin{displaymath}
    \kappa^{(0)}_{B_1}:=\min\{h\in\N_0:\whJ_h=B_1\},
    \quad \kappa^{(i)}_{B_1}:=\min\{h> \kappa^{(i-1)}_{B_1}:\whJ_h=B_1\},\quad i\in\N.
  \end{displaymath}
  Then $\P\{\kappa^{(0)}_{B_1}<\infty\}=1$ and
  $(\kappa^{(i)}_{B_1}-\kappa^{(i-1)}_{B_1})_{i\in\N}$ are
  independent identically distributed with
  \begin{displaymath}
    \E \left(\kappa^{(i)}_{B_1}-\kappa^{(i-1)}_{B_1}\right)<\infty,\quad i\in\N.
  \end{displaymath}
\end{proposition}
\begin{proof}
  Let 
  \begin{displaymath}
    r_h:=\sup\{t\geq 0: \whJ_h\ominus B_{t}\neq \varnothing\},\quad h\in\N_0,
  \end{displaymath}
  be the radius of the largest ball inscribed in $\whJ_h$. By 
  Proposition~\ref{prop:contains_a_ball}, $r_h\in [r,1]$
  for all $h\in\N_0$. Fix $\delta\in (\tau,1)$ and define the
  events
  \begin{displaymath}
    A_h(\delta):=\{\sU_h(\whJ_h)\in \whJ_h\ominus B_{\delta r_h}\}
    =\{B_{\delta r_h}\subset \whJ_h-\sU_h(\whJ_h)\},\quad h\in\N_0.
  \end{displaymath}
  Note that, for $t\in (0,1]$,
  \begin{align*}
    \{r_{h+m+1}\geq t\}&=\{\whJ_{h+m+1}\ominus B_{t}\neq\varnothing\}\\
    &\supset \{(\tau^{-1}(\whJ_{h+m}-\sU_{h+m}(\whJ_{h+m}))\cap 
    B_1)\ominus B_{t}\neq\varnothing\}\cap A_{h+m}(\delta)\\
    &\supset \{(\tau^{-1}B_{\delta r_{h+m}}\cap 
    B_1)\ominus B_{t}\neq\varnothing\}\cap A_{h+m}(\delta)\\
    &=\{\tau^{-1}\delta r_{h+m}\geq t\}\cap A_{h+m}(\delta).
  \end{align*}
  Iterating this inclusion we arrive at
  \begin{equation}\label{eq:pos_recurrence_proof1}
    \{r_{h+m+1}\geq t\}\supset
    \{r_{h}\geq t(\tau/\delta)^{m+1}\}\cap \bigcap_{k=0}^{m}A_{h+k}(\delta),\quad h,m\in\N_0.
  \end{equation}
  Define $m_0:=\inf\{k\in\N_0:(\tau/\delta)^{k+1}\leq r\}$ and
  plug $t=1$ into~\eqref{eq:pos_recurrence_proof1}. Since $r_h\geq r$,
  this yields
  \begin{equation}\label{eq:pos_recurrence_proof2}
    \{\whJ_{h+m_0+1}=B_1\}=\{r_{h+m_0+1}=1\}
    \supset \bigcap_{k=0}^{m_0}A_{h+k}(\delta),\quad h\in\N_0.
  \end{equation}
  Suppose that 
  \begin{equation}\label{eq:pos_recurrence_proof3}
    \P\left\{\bigcap_{k=0}^{m_0}A_{h+k}(\delta)
      \Big| \whJ_h\in \cdot\right\}\geq p^{\ast},\quad h\in\N_0,
  \end{equation}
  for a positive constant $p^{\ast}$. 
  Inclusion~\eqref{eq:pos_recurrence_proof2} demonstrates that from
  any state $\whJ_h$ with probability at least $p^{\ast}$ the chain
  $(\whJ_h)$ visits the state $B_1$ after exactly $m_0+1$
  steps. Dividing the entire trajectory of $(\whJ_h)$ into consecutive
  blocks of size $m_0+1$, we see that the distribution of
  $\kappa_{B_1}$ (conditional, given $\whJ_0=B_1$) is stochastically
  dominated by the product of the constant $(m_0+1)$ and a geometrically
  distributed random variable with success probability $p^{\ast}$.

  It remains to confirm~\eqref{eq:pos_recurrence_proof3}. We have
  \begin{equation}\label{eq:pos_recurrence_proof4}
    \P\left\{\bigcap_{k=0}^{m_0}A_{h+k}(\delta)
      \Big| \whJ_h\in \cdot\right\}
    =\prod_{k=0}^{m_0}\P\left\{A_{h+k}(\delta)
      \big|\cap_{j=0}^{k-1}A_{h+j}(\delta),\whJ_h\in\cdot\right\}.
  \end{equation}
  Let $\sigma_h$ be the $\sigma$-algebra generated by
  $\{\sU_0,\sU_1,\ldots,\sU_{h-1}\}$, $h\in\N$. By the construction,
  $\whJ_h$ is $\sigma_h$-measurable and $A_{h}(\delta)$ belongs to
  $\sigma_{h+1}$. Therefore,
  $\cap_{j=0}^{k-1}A_{h+j}(\delta)\cap\{\whJ_h\in\cdot\}$ is
  $\sigma_{h+k}$-measurable.

  Let $A$ be an arbitrary event from $\sigma_{h+k}$.
  By the definition of $r_{h+k}$,
  there exists a $\sigma_{h+k}$-measurable point $c_{h+k}$ such that
  $B_{(1-\delta) r_{h+k}}(c_{h+k})\subset \whJ_{h+k}\ominus
  B_{\delta r_{h+k}}$. Since $\sU_{h+k}$ is
  independent of $\sigma_{h+k}$ and $(\whJ_{h+k}, r_{h+k}, c_{h+k})$
  is $\sigma_{h+k}$-measurable, 
  \begin{align*}
    \P\big\{A_{h+k}(\delta)\,|\,A\big\}
    &=\P\big\{A_{h+k}(\delta)\,|\,A\big\}
      =\P\big\{\sU_{h+k}(\whJ_{h+k})\in
      \whJ_{h+k}\ominus B_{\delta r_{h+k}}\,|\,A\big\}\\
    &\geq \P\big\{\sU_{h+k}(\whJ_{h+k})\in B_{(1-\delta)
      r_{h+k}}(c_{h+k})\,|\,A\big\}\\
    &=\frac{1}{\P\{A\}}\E\left(\E\left(\frac{\lambda\big(B_{(1-\delta)
      r_{h+k}}(c_{h+k})\big)}{\lambda(\whJ_{h+k})}\1_{A}\Big|\sigma_{h+k}\right)\right)\\
    &\geq \frac{1}{\P\{A\}}\E\left(\E\left(
      \frac{\lambda\big(B_{(1-\delta)r}\big)}{\lambda(B_1)}
      \1_{A}\Big|\sigma_{h+k}\right)\right)\\
    &=\frac{\lambda(B_{(1-\delta)r})}{\lambda(B_1)}=\big((1-\delta)r\big)^d>0.
  \end{align*}
  This bound implies~\eqref{eq:pos_recurrence_proof3} in view
  of~\eqref{eq:pos_recurrence_proof4}.
\end{proof}

\begin{proof}[Proof of Theorem~\ref{thm:main1}]
  \am{Note that the chain $(\whJ_h)_{h\in\N_0}$ is aperiodic by~\eqref{eq:aperiodicity}.
  By Proposition~\ref{prop:pos_recurrence} the chain $(\whJ_h)_{h\in\N_0}$ is recurrent and 
  its return time to the state $B_1$ is integrable. By Theorems 6.8.5 and 6.8.7 in~\cite{Durrett}, the 
  stationary measure of $(\whJ_h)_{h\in\N_0}$ exists and up to a multiplicative constant is given by
  \begin{displaymath}
    \mu(\cdot) = \E\left(\sum_{h\geq 0}\1_{\{\whJ_{h}\in\cdot,\kappa^{(0)}_{B_1}\leq h<\kappa^{(1)}_{B_1}\}}\right).
  \end{displaymath}
  By Theorem 6.8.8 in~\cite{Durrett}, the sequence of probability measures $(\mathbb{P}\{\whJ_h \in \cdot\})_{h\in\N_0}$ converges, as $h\to\infty$, in the total variation distance to the unique stationary probability distribution 
  $$
  \pi_{\infty}(\cdot):=\mu(\cdot)/\mu(\sK^d)=\frac{1}{\E\left(\kappa^{(1)}_{B_1}-\kappa^{(0)}_{B_1}\right)}\E\left(\sum_{h\geq 0}\1_{\{\whJ_{h}\in\cdot,\kappa^{(0)}_{B_1}\leq h<\kappa^{(1)}_{B_1}\}}\right).
  $$}
 
  The fact that $X_{\infty}$ is a finite intersection of balls with
  radii $\{1,\tau^{-1},\tau^{-2},\ldots\}$ with probability one
  follows from~\eqref{eq:contains_a_ball_main_representation} and
  $\P\big\{\kappa^{(1)}_{B_1}-\kappa^{(0)}_{B_1}<\infty\big\}=1$.
\end{proof}

\section{The length of the leftmost path}\label{sec:leftmost}

Recall that the left boundary of a vp tree $\VP(x_1,x_2,\ldots)$ is
the unique path which starts at the root and on each step follows the
left subtree. Also recall the notation $(x_{l_h},r_{x_{l_h}})$ for the
vertex of depth $h\in\{0,1,2,\ldots\}$ and its threshold in this path and
$(I_{h})_{h\in\N_0}$ for the sequence defined in~\eqref{eq:I_h_def}.

We are interested in the number $L_n$ of edges in the leftmost path of
$\VP(\bU_1,\bU_2,\ldots,\bU_n)$ with exponential threshold
function~\eqref{eq:r_x_exponential}. Let $l_1-l_0$ be the number of
trials (insertions of new vertices) until a left child is attached to
the root. Obviously, given $I_1$, $l_1-l_0$ has a geometric law on
$\N$ with success probability $\lambda(I_1)/\lambda(K)$. Similarly,
given $(I_h)_{h=0,\ldots,k}$, $l_k-l_{k-1}$ has a geometric law on
$\N$ with success probability $\lambda(I_k)/\lambda(K)$, and
$l_1-l_0,l_2-l_1,\ldots,l_k-l_{k-1}$ are conditionally
independent. According to Proposition~\ref{prop:markovian_redef} and
the discussion afterwards, the distribution of the sequence
$(l_k-l_{k-1})_{k\in\N}$ is the same as that of the sequence
$(G_k)_{k\in\N}$ comprised of conditionally independent, given
$(\wtJ_h)_{h\in\N_0}$, random variables such that
\begin{displaymath}
  \P\{G_k=j\,|\,\wtJ_0,\ldots,\wtJ_k\}
  =\frac{\lambda(\wtJ_k)}{\lambda(K)}
  \left(1-\frac{\lambda(\wtJ_k)}{\lambda(K)}\right)^{j-1},\quad j\in\N,\quad k\in\N.
\end{displaymath}
Put $S_0:=0$ and $S_k:=G_1+G_2+\cdots+G_k$, $k\in\N$. Notice that the
sequence $(1+S_k)_{k\in\N_0}$ is distributed as the sequence of time
epochs when new vertices are attached to the leftmost path. Thus, see
also Eq.~(24) in~\cite{Bohun},
\begin{equation}\label{eq:duality}
  L_n\overset{d}{=}\max\{k\in\N_0:1+S_k\leq n\},\quad n\in\N.
\end{equation}
To derive a limit theorem for $L_n$ we start with a couple of
lemmas.

\begin{lemma}\label{lem:leftmost1}
  For every fixed $l\in\N_0$, we have
  \begin{displaymath}
    \tau^{-n}\left(\wtJ_n,\wtJ_{n-1},\ldots,\wtJ_{n-l}\right)
    ~\todistr~(J_{\infty}^{(0)},J_{\infty}^{(1)},\ldots,J_{\infty}^{(l)})
    \quad \text{as}\; n\to\infty.
  \end{displaymath}
  The limit sequence $(J_{\infty}^{(h)})_{h\in \N_0}$ is defined as
  follows: $(\tau^{h}J_{\infty}^{(h)})_{h\in \N_0}$ is a stationary
  sequence of consecutive values of a Markov chain~\eqref{eq:J_h_def2}
  which starts at the stationary distribution $X_{\infty}$ defined in
  Theorem~\ref{thm:main1}.
\end{lemma}
\begin{proof}
  Follows immediately from Theorem~\ref{thm:main1}.
\end{proof}

It is known that the volume mapping $\lambda:\sK_d\mapsto [0,\infty)$
is continuous with respect to the Hausdorff metric, see Theorem~1.8.20
in~\cite{Schneider}. Therefore, for every fixed $l\in\N_0$,
\begin{equation*}
  \tau^{-dn}\left(\frac{\lambda(\wtJ_n)}{\lambda(K)},
    \frac{\lambda(\wtJ_{n-1})}{\lambda(K)},\ldots,
    \frac{\lambda(\wtJ_{n-l})}{\lambda(K)}\right)~\todistr~
  \left(\frac{\lambda(J_{\infty}^{(0)})}{\lambda(K)},
    \frac{\lambda(J_{\infty}^{(1)})}{\lambda(K)},\ldots,
    \frac{\lambda(J_{\infty}^{(l)})}{\lambda(K)}\right)
  \quad \text{as}\;n\to\infty.
\end{equation*}
Given, $(J_{\infty}^{(h)})_{h\in \N_0}$, let
$(\mathcal{E}_l)_{l\in\N_0}$ be a sequence of conditionally
independent random variables with the exponential distributions
\begin{displaymath}
  \P\{\mathcal{E}_l\geq t\,|\,(J_{\infty}^{(h)})_{h\geq 0}\}
  =\exp\left(-t\frac{\lambda(J_{\infty}^{(l)})}{\lambda(K)}\right),
  \quad t\geq 0,\quad l\in\N_0.
\end{displaymath}

\begin{lemma}\label{lem:leftmost2}
  The random series $S_{\infty}:=\sum_{l=0}^{\infty}\mathcal{E}_l$
  converges a.s. and in mean.
\end{lemma}
\begin{proof}
  The claims follow from
  \begin{multline*}
    \sum_{l=0}^{\infty}\E (\mathcal{E}_l)
    =\sum_{l=0}^{\infty}\E (\E
    (\mathcal{E}_l|(J_{\infty}^{(h)})_{h\geq 0}))
    =\lambda(K)\sum_{l=0}^{\infty}\E \left(\frac{1}{\lambda(J_{\infty}^{(l)})}\right)\\
    =\lambda(K)\sum_{l=0}^{\infty}\tau^{ld}
    \E \left(\frac{1}{\lambda(\tau^{l}J_{\infty}^{(l)})}\right)
    =\lambda(K)\E \left(\frac{1}{\lambda(X_{\infty})}\right)\frac{1}{1-\tau^d}<\infty,
  \end{multline*}
  where we have used stationarity of
  $(\tau^{h}J_{\infty}^{(h)})_{h\in \N_0}$. The fact that
  $\E \left(1/\lambda(X_{\infty})\right)<\infty$ is a
  consequence of Proposition~\ref{prop:contains_a_ball} which implies
  that $\lambda(X_{\infty})$ is bounded away from zero.
\end{proof}

\begin{lemma}\label{lem:leftmost3}
  As $n\to\infty$, it holds
  \begin{equation}\label{eq:lemma_leftmost_claim}
    \tau^{dn}S_n~\todistr~\sum_{l=0}^{\infty}\mathcal{E}_l=S_{\infty}.
  \end{equation}
\end{lemma}
\begin{proof}
  According to Proposition~\ref{prop:contains_a_ball}, there exist
  $0<c_1<c_2<\infty$ such that
  \begin{equation}\label{eq:two_sided_bound_volumes}
    \P\left\{c_1\tau^{dh}\leq
      \frac{\lambda(\wtJ_h)}{\lambda(K)}\leq c_2\tau^{dh},h\in\N_0\right\}=1.
  \end{equation}
  Put
  $Z_n(t):=\E
  \left(e^{it\tau^{dn}S_n}\,|\,\wtJ_0,\ldots,\wtJ_n\right)$. It
  suffices to show that, for every fixed $t\in\R$,
  \begin{equation}\label{eq:lemma_leftmost_proof1}
    Z_n(t)~\todistr~\prod_{h=0}^{\infty}
    \frac{1}{1-\lambda(K)it/\lambda(J_{\infty}^{(h)})}
    \quad \text{as}\; n\to\infty.
  \end{equation}
  The convergence~\eqref{eq:lemma_leftmost_proof1}
  yields~\eqref{eq:lemma_leftmost_claim} by the Lebesgue dominated
  convergence theorem.

  Let $\log$ denote the principal branch of the complex
  logarithm. For fixed $t\in\R$ and
  using~\eqref{eq:two_sided_bound_volumes}, we obtain
  \begin{align*}
    -\log Z_n(t)
    &=\sum_{h=1}^{n}\log
      \left(1-\frac{\lambda(K)}{\lambda(\wtJ_h)}(1-e^{-it\tau^{nd}})\right)\\
    &=\sum_{h=1}^{n}\log \left(1-\frac{\lambda(K)}{\lambda(\wtJ_h)}
      (it\tau^{nd}+O(\tau^{2nd}))\right)
      =\sum_{h=1}^{n}\log \left(1-\frac{\lambda(K)}{\lambda(\wtJ_h)}
      it\tau^{nd}\right)+o(1),
  \end{align*}
  where $o(1)$ is a non-random sequence which converges to zero as
  $n\to\infty$. Further,
  \begin{align*}
    \sum_{h=1}^{n}\log \left(1-\frac{\lambda(K)}{\lambda(\wtJ_h)}
    it\tau^{nd}\right)
    &=\sum_{h=0}^{n-1}\log \left(1-\frac{\lambda(K)}
      {\tau^{-nd}\lambda(\wtJ_{n-h})}it\right)\\
    &=\left(\sum_{h=0}^{M}+\sum_{h=M+1}^{n-1}\right)
      \log \left(1-\frac{\lambda(K)}{\tau^{-nd}\lambda(\wtJ_{n-h})}
      it\right)=:A_{n,M}(t)+B_{n,M}(t).
  \end{align*}
  By Lemma~\ref{lem:leftmost1} we have
  \begin{displaymath}
    A_{n,M}(t)~\todistr~\sum_{h=0}^{M}
    \log \left(1-\lambda(K)it/\lambda(J_{\infty}^{(h)})\right),\quad n\to\infty,
  \end{displaymath}
  for every fixed $M\in\N$. As $M\to\infty$, $e^{-A_{n,M}(t)}$
  converges to the right-hand side
  of~\eqref{eq:lemma_leftmost_proof1}. According to Theorem 3.2
  in~\cite{Billingsley} it remains to check that, for every fixed
  $\eps>0$,
  \begin{equation}\label{eq:bill1}
    \lim_{M\to\infty}\limsup_{n\to\infty}
    \P\big\{|B_{n,M}(t)|\geq \eps\,|\,\wtJ_0,\ldots,\wtJ_n\}=0
    \quad\text{a.s.}
  \end{equation}
  Using~\eqref{eq:two_sided_bound_volumes} we infer, for some $C>0$,
  \begin{align*}
    |B_{n,M}(t)|
    \leq \sum_{h=M+1}^{n-1}\left|
    \log \left(1-\frac{\lambda(K)}{\tau^{-nd}\lambda(\wtJ_{n-h})}
    it\right)\right|
    &\leq C|t|\sum_{h=M+1}^{n-1}\frac{\lambda(K)}{\tau^{-nd}\lambda(\wtJ_{n-h})}\\
    &\overset{\eqref{eq:two_sided_bound_volumes}}{\leq}
      Cc_1^{-1}|t|\sum_{h=M+1}^{n-1}\tau^{hd}.
  \end{align*}
  This clearly implies~\eqref{eq:bill1} and the proof is complete.
\end{proof}

Combining the above lemmas and the duality relation~\eqref{eq:duality}
we arrive at the following result.

\begin{theorem}\label{thm:main2}
  Under the same assumptions as in Theorem~\ref{thm:main1}, for every
  fixed $x>0$, it holds
  \begin{displaymath}
    \lim_{n\to\infty}\P\big\{L_{\lfloor x\tau^{-nd}\rfloor}\leq
    n+s\big\}
    =\P\big\{S_{\infty}\geq x\tau^{sd}\big\}
    =\P\left\{\frac{\log S_{\infty}-\log x}{d\log\tau}\leq s\right\},
    \quad s\in\Z,
  \end{displaymath}
  where $S_{\infty}$ is defined in Lemma~\ref{lem:leftmost2}.
\end{theorem}
\begin{proof}
  Fix $s\in\Z$, $x>0$ and write
  \begin{align*}
    \P\big\{L_{\lfloor x\tau^{-nd}\rfloor}>n+s\big\}
    =\P\big\{1+S_{n+s}\leq \lfloor x\tau^{-nd}\rfloor\big\}
    =\P\big\{\tau^{(n+s)d}+\tau^{(n+s)d}S_{n+s}
    \leq \lfloor x\tau^{-nd}\rfloor\tau^{(n+s) d}\big\}.
  \end{align*}
  It is easy to check that the distribution of $S_{\infty}$ is
  continuous. Thus, letting $n\to\infty$ yields that the right-hand
  side converges to
  $\P\{S_{\infty}\leq x\tau^{sd}\}=\P\{S_{\infty}< x\tau^{sd}\}$ by
  Lemma~\ref{lem:leftmost3}.
\end{proof}

\begin{corollary}
  The following weak laws of large numbers hold 
  \begin{displaymath}
    \frac{\log S_n}{n}~\toprobab~d\log (1/\tau)
    \quad \text{as}\; n\to\infty,
  \end{displaymath}
  and
  \begin{equation}\label{eq:Ln_wlln}
    \frac{L_n}{\log n}~\toprobab~\frac{1}{d\log (1/\tau)}
    \quad \text{as}\; n\to\infty.
  \end{equation}
\end{corollary}

Let $H_n$ be the height of $\VP(\bU_1,\bU_2,\ldots,\bU_n)$ which is the length of the longest path from the root to a leaf. \am{The height is one of the simplest yet most important functionals on trees, as it corresponds to the recursion depth required for traversing the tree using a straightforward recursive algorithm: the algorithm visits the root first, then recursively processes each subtree of the root. The exact asymptotic behavior of the random sequence $(H_n)_{n\in\N}$ for large $n$ remains unknown, even in the simplest metric space $K=[-1,1]$ with $\ell_{\infty}$ norm. However, since $L_n\leq H_n$, our Theorem~\ref{thm:main2} implies
  \begin{displaymath}
    \lim_{n\to\infty}\P\left\{H_n\geq
      \left(\frac{1}{d\log (1/\tau)}-\eps\right)\log n\right\}=1,
  \end{displaymath}
for every fixed $\eps>0$. Simulations carried out in~\cite{Bohun_thesis} for the case $K=[-1,1]$, see Section 3.5 therein, suggest that the quantity $H_n/L_n\geq 1$ converges to a finite limit at least for some values of $\tau\in (0,1)$. We believe that this holds also in the general settings studied in this paper, namely we {\it conjecture} that
  \begin{equation}\label{eq:Hn_wlln}
    \frac{H_n}{\log n}~\toprobab~H_{\infty}(\tau,d),\quad n\to\infty,
  \end{equation}
  for some finite deterministic constant $H_{\infty}(\tau,d)\geq
  \frac{1}{d\log (1/\tau)}$.  
\begin{Problem} 
Determine the values of $\tau$ and $d$ such that relation~\eqref{eq:Hn_wlln} holds true and find the constant $H_{\infty}(\tau,d)$.
\end{Problem}}

\backmatter

\bmhead{Acknowledgements}
This work has been accomplished during AM's visit to Queen Mary
University of London as Leverhulme Visiting Professor in July-December
2023. The research of CD and AM was partially supported by the High
Level Talent Project DL2022174005L of Ministry of Science and
Technology of PRC. The authors are grateful to the referees for
careful reading of this work. 

Data Availability: No datasets were generated or analysed during the current study.

\end{document}